\documentclass[12pt,leqno]{amsart}
\usepackage{amsmath,amsfonts,amssymb,amsthm}
\usepackage{graphicx}
\graphicspath{ }
\usepackage{wrapfig}
\usepackage{accents}
\usepackage{caption}
\usepackage{subcaption}
\usepackage{calligra}

\numberwithin{figure}{section}

\theoremstyle{plain}
\newtheorem{thm}{Theorem}[section]

\newtheorem{cor}{Corollary}[thm]
\theoremstyle{definition}
\newtheorem{defn}{Definition}[section]

\numberwithin{equation}{section}
\theoremstyle{remark}

\usepackage{mathtools}
\title{Some characterizations of Rectifying and osculating curves on a smooth immersed surface}
\author[A. A. Shaikh and P. R. Ghosh]{Absos Ali Shaikh$^*$ and Pinaki Ranjan Ghosh}
\address{\noindent\newline  Department of Mathematics,\newline University of
Burdwan, Golapbag,\newline Burdwan-713104,\newline West Bengal, India}
\email{aask2003@yahoo.co.in, aashaikh@math.buruniv.ac.in}
\address{\noindent\newline  Department of Mathematics,\newline University of
Burdwan, Golapbag,\newline Burdwan-713104,\newline West Bengal, India}
\email{mailtopinaki94@gmail.com}
\begin{document}

\begin{abstract}
The present paper deals with some characterizations of rectifying and osculating curves on a smooth surface with respect to the reference frame $\{\vec{T},\ \vec{N},\ \vec{T}\times\vec{N}\}$. We have computed the components of position vectors of rectifying and osculating curves along $\vec{T},\ \vec{N},\ \vec{T}\times\vec{N}$ and then investigated their invariancy under isometry of surfaces, and it is shown that they are invariant iff either the normal curvature of the curve is invariant or the position vector of the curve is in the direction of the tangent vector to the curve.
\end{abstract}
\noindent\footnotetext{ $^*$ Corresponding author.\\
$\mathbf{2010}$\hspace{5pt}Mathematics\; Subject\; Classification: 53C22, 53A04, 53A05. }
\maketitle
\section{Introduction}
When we discuss about a smooth curves in the Euclidean space $\mathbb{R}^3$, one of the beautiful tool is Serret-Frenet frame, consists of three mutually orthogonal unit vectors $\{\vec{t},\ \vec{n},\ \vec{b}\}$ such that $\{\vec{t},\ \vec{n}\}$, $\{\vec{n},\ \vec{b}\}$ and $\{\vec{b},\ \vec{t}\}$ are respectively spanned three mutually orthogonal planes named as osculating, normal and rectifying plane in the Euclidean space $\mathbb{R}^3$. Now curves with position vectors lie in the above defined three planes are respectively called osculating, normal and rectifying curves. But whenever we consider to study a curve on a smooth surface immersed in $\mathbb{R}^3$, at every point of the curve another three orthogonal unit vectors $\{\vec{T},\ \vec{N},\ \vec{T}\times\vec{N}\}$ comes naturally, where $\vec{T}$ is any tangent vector to the surface at that point and $\vec{N}$ being the normal to the surface.
\par
Bang-Yen Chen (\cite{BYC03}) introduced and characterized the rectifying curves in $\mathbb{R}^3$ and also studied (\cite{BYC05}) the relationship among rectifying curves, centrodes and extremal curves. For more interesting properties of osculating, normal and rectifying curves we refer the reader to see \cite{IN08}, \cite{BYC18} and \cite{ISU Novi Sad}.

\par
In $2011$ Camci et. al. (\cite{CKI11}) first studied curves on a surface immersed in $\mathbb{R}^3$ whose position vectors lie in the plane spanned by $\{\vec{T},\ \vec{N}\}$, $\{\vec{N},\ \vec{T}\times\vec{N}\}$, $\{\vec{T},\ \vec{T}\times\vec{N}\}$ and characterized such curves. The present authors (\cite{PRG18A}, \cite{PRG18B}, \cite{PRG19B}) studied rectifying, osculating and normal curves on a smooth immersed surface in $\mathbb{R}^3$ and obtained their characterizations under isometry of surfaces. The present authors (\cite{PRG19A}) also studied curves on a surface whose position vectors lie in its tangent plane.
\par
Motivating by the above studies, the objective of this paper is to study rectifying and osculating curves on a smooth immersed surface with respect to the orthogonal frame $\{\vec{T},\ \vec{N},\ \vec{T}\times\vec{N}\}$. In section 2 we discuss some basic definitions and elementary results. Section 3 and 4 are respectively deal with the characterization of rectifying and osculating curves on a smooth surface with respect to the reference frame $\{\vec{T},\ \vec{N},\ \vec{T}\times\vec{N}\}$. 
We obtain a sufficient condition for which component of rectifying curves along $\vec{T}$ and $\vec{T}\times\vec{N}$ respectively to be invariant under rectifying curve preserving isometry of surfaces (see Theorem 3.1., Theorem 3.2.). 
It is also shown that the component of the position vector of an osculating curve along $\vec{T}\times\ \vec{N}$ is invariant under osculating curve preserving isometry (see Theorem $4.2.$).
\section{preliminaries}
Let $\gamma(s)$ be an arc length parametrized curve on a surface patch $\phi(u,v)$ of a smooth surface $S$. The tangent, normal and binormal vector to the curve $\gamma(s)$ at any point $\gamma(s)$ are respectively given by 
\begin{eqnarray}
\nonumber \vec{t}(s)&=&\phi_uu'+\phi_vv',\\
\nonumber
\vec{n}(s)&=&\frac{1}{\kappa(s)}(\phi_uu''+\phi_vv''+\phi_{uu}u'^2+2\phi_{uv}u'v'+\phi_{vv}v'^2),\\
\nonumber
\vec{b}(s)&=& \frac{1}{\kappa(s)}\Big[\{v''u'-u''v'\}\vec{N}+u'^3(\phi_u\times \phi_{uu})+2u'^2v'(\phi_u\times \phi_{uv})\\
\nonumber
&&+u'v'^2(\phi_u\times \phi_{vv})+u'^2v'(\phi_v\times \phi_{uu})+2u'v'^2(\phi_v\times \phi_{uv})+v'^3(\phi_v\times \phi_{vv})\Big].
\end{eqnarray}
We note that $\gamma(s)$ is a rectifying (respectively, osculating) curve if $\gamma(s)=\lambda\vec{t}+\mu\vec{b}$ (respectively, $\alpha(s)=\lambda\vec{t}+\mu\vec{n}$), where $\lambda$, $\mu$ are smooth functions.
\begin{defn}
Two surfaces $S$ and $\bar{S}$ are isometric if there exists a diffeomorphism $f:S\rightarrow\bar{S}$ such that the length of any curve preserves under $f$. 
\end{defn}
\par
\begin{defn}
Let $\gamma(s)$ be a curve parametrized by arc length on a smooth surface $S$ in $\mathbb{R}^3$. Then $\gamma''(s)$ lies in the plane perpendicular to $\gamma'(s)$, i.e., lies in the plane spanned by the vectors $\vec{N}$ and $\gamma'(s)\times\vec{N}$. So we have two non-zero components of $\gamma''(s)$ along $\vec{N}$ and $\gamma'(s)\times\vec{N}$. Now 
\begin{eqnarray*}
\kappa_g(s)=\gamma''(s)\cdot(\gamma'(s)\times\vec{N}),\\
\kappa_n(s)=\gamma''(s)\cdot\vec{N},
\end{eqnarray*}
where $\kappa_g(s)$ and $\kappa_n(s)$ are respectively known as geodesic and normal curvature of $\gamma(s)$. Since $\gamma''(s)=\kappa(s)\vec{n}(s)$, hence $\kappa_n(s)=Lu'^2+2Mu'v'+Nv'^2$, where $L$, $M$ and $N$ are coefficients of second fundamental form of the surface. For more information of isometry, fundamental forms, normal curvature we refer the reader to see \cite{AP01} and \cite{MPDC76 }.
\end{defn}
\section{rectifying curves on a surface}
It is well-known that at every point of any smooth curve in $\mathbb{R}^3$ there exists a Serret-Frenet frame. Now if the curve lies on a surface $S$ then there exists another frame of reference $\{\vec{T},\vec{N},\vec{T}\times\vec{N}\}$ at each point of the curve related to the surface. Throughout this section and the next section we will deduce the components of position vectors of rectifying and osculating curves along $\vec{T},\ \vec{N},\ \vec{T}\times\vec{N}$ and then investigate their invariancy under surface isometry.
The equation of a rectifying curve on $S$ is given by
\begin{eqnarray}\label{r1}
\nonumber
\gamma(s)&=&\lambda(s)(\phi_uu'+\phi_vv')+\frac{\mu(s)}{\kappa(s)}\Big[u'^3(\phi_u\times \phi_{uu})+2u'^2v'(\phi_u\times \phi_{uv})+u'v'^2(\phi_u\times \phi_{vv})\\
&&+u'^2v'(\phi_v\times \phi_{uu})+2u'v'^2(\phi_v\times \phi_{uv})+v'^3(\phi_v\times \phi_{vv})+(v''u'-v'2u'')\vec{N}\Big],
\end{eqnarray}
for some functions $\lambda(s)$ and $\mu(s)$.
\begin{thm}
 Let $f:S\rightarrow\bar{S} $ be an isometry. If $\gamma$ and $\bar{\gamma}$ are  rectifying curves on $S$ and $\bar{S}$ respectively, then for the component of $\gamma(s)$ along any tangent vector $\vec{T}=a\phi_u+b\phi_v,\ a,\ b\in \mathbb{R}$, to the surface $S$ at $\gamma(s)$, the following holds:
 \begin{eqnarray}\label{rt1} 
 \bar{\gamma}\cdot \vec{\bar{T}}-\gamma\cdot \vec{ T}=\frac{\mu(s)}{\kappa(s)}(av'+bu')\Big(\bar{\kappa}_n(s)-\kappa_n(s)\Big).
 \end{eqnarray}
\end{thm}
\begin{proof}
From $(\ref{r1})$, we see that
\begin{equation}\label{rt2}
\gamma\cdot\phi_u=\lambda(Eu'+Fv')+\frac{\mu(s)}{\kappa(s)}\Big[Lu'^2 v' +2Mv'^2 u'+Nv'^3 \Big].
\end{equation}
Similarly for the rectifying curve $\bar{\gamma}(s)$ on $\bar{S}$, we have
\begin{equation*}
\bar{\gamma}\cdot\bar{\phi}_u=\lambda(\bar{E}u'+\bar{F}v')+\frac{\mu(s)}{\kappa(s)}\Big[\bar{L}u'^2 v' +2\bar{M}v'^2 u'+\bar{N}v'^3 \Big].
\end{equation*}
Since $f$ is isometry, hence $\bar{E}=E$, $\bar{F}=F$ and $\bar{G}=G$. Therefore
\begin{equation}\label{rt3}
\bar{\gamma}\cdot\bar{\phi}_u=\lambda(Eu'+Fv')+\frac{\mu(s)}{\kappa(s)}\Big[\bar{L}u'^2 v' +2\bar{M}v'^2 u'+\bar{N}v'^3\Big].
\end{equation}
Taking the difference of $(\ref{rt2})$ and $(\ref{rt3})$, we get
\begin{equation}\label{rt4}
\bar{\gamma}\cdot\bar{\phi}_u-\gamma\cdot\phi_u=\frac{v'\mu(s)}{\kappa(s)}\Big(\bar{\kappa}_n(s)-\kappa_n(s)\Big).
\end{equation}
Similarly the following relation holds:
\begin{equation}\label{rt5}
\bar{\gamma}\cdot\bar{\phi}_v-\gamma\cdot\phi_v=\frac{u'\mu(s)}{\kappa(s)}\Big(\bar{\kappa}_n(s)-\kappa_n(s)\Big).
\end{equation}
Now with the help of $(\ref{rt4})$ and $(\ref{rt5})$ we get
\begin{eqnarray*}
\bar{\gamma}\cdot\vec{\bar{T}}-\gamma\cdot T &=&\bar{\gamma}\cdot(a\bar{\phi}_u+b\bar{\phi}_v)-\gamma\cdot(a\phi
_u+b\phi
_v),\\
&=& a(\bar{\gamma}\cdot \bar{\phi}_u -\gamma\cdot \phi_u)+b(\bar{\gamma}\cdot \bar{\phi}_v -\gamma\cdot \phi_v),\\
&=& a\frac{v'\mu(s)}{\kappa(s)}\Big(\bar{\kappa}_n(s)-\kappa_n(s)\Big)+b\frac{u'\mu(s)}{\kappa(s)}\Big(\bar{\kappa}_n(s)-\kappa_n(s)\Big),\\
&=&
\frac{\mu(s)}{\kappa(s)}(av'+bu')\Big(\bar{\kappa}_n(s)-\kappa_n(s)\Big).
\end{eqnarray*}
This proves our claim.
\end{proof}
 \begin{thm}
  Let $f:S\rightarrow\bar{S} $ be an isometry. If $\gamma$ and $\bar{\gamma}$ are  rectifying curves on $S$ and $\bar{S}$ respectively, then for the component of $\gamma(s)$ along $\vec{T}\times\vec{N}$, the following holds:
 \begin{equation}\label{rt6}
 \bar{\gamma}\cdot(\vec{\bar{T}}\times\vec{\bar{N}})-\gamma\cdot(\vec{T}\times\vec{N})=\frac{\mu(s)}{\kappa(s)}\Big(\bar{\kappa}_n(s)-\kappa_n(s)\Big)\Big\{vaF-uaEb+vbG-ubF\Big\}.
 \end{equation}
 \end{thm}
 \begin{proof}
 The component of $\gamma(s)$ along $\vec{T}\times\vec{N}$ is given by
 \begin{eqnarray*}
 \gamma\cdot(\vec{T}\times\vec{N})&=& \gamma\cdot\{(a\varphi_u+b\varphi_v)\times\vec{N}\},\\
 &=& \gamma\cdot\{Fa\phi_u-Ea\phi_v+Gb\phi_u-Fb\phi_v\},\\
 &=& (Fa+Gb)\gamma\cdot\phi_u-(Ea+Fb)\gamma\cdot\phi_v.
 \end{eqnarray*}
 Hence 
 \begin{equation*}
 \bar{\gamma}\cdot(\vec{\bar{T}}\times\vec{\bar{N}})-\gamma\cdot(\vec{T}\times\vec{N})=(Fa+Gb)(\bar{\gamma}\cdot\bar{\phi}_u-\gamma\cdot\phi_u)-(Ea+Fb)(\bar{\gamma}\cdot\bar{\phi}_v-\gamma\cdot\phi_v).
 \end{equation*}
 Now using $(\ref{rt4})$ and $(\ref{rt5})$ we get
 \begin{equation}
 \bar{\gamma}\cdot(\vec{\bar{T}}\times\vec{\bar{N}})-\gamma\cdot(\vec{T}\times\vec{N})=\frac{\mu(s)}{\kappa(s)}\Big(\bar{\kappa}_n(s)-\kappa_n(s)\Big)\Big\{vaF-uaE+vbG-ubF\Big\}.
 \end{equation}
 This proves the result.
 \end{proof}
 \begin{cor}\label{rcor2}
  Let $f:S\rightarrow \bar{S}$ be an isometry and $\gamma(s)$ be an rectifying curve on $S$. Then the component of the rectifying curve $\gamma(s)$ along $\vec{T}\times\vec{N}$ (respectively, $\vec{T}$) is invariant iff any one of the following holds:
  \begin{itemize}
  \item[(i)]the position vector of $\gamma(s)$ is in the direction of tangent vector to $\gamma$.
  \item[(ii)] The normal curvature is invariant.
  \end{itemize}
  \end{cor}
  \begin{proof}
  From $(\ref{rt6})$,  $\bar{\gamma}(s)\cdot(\vec{\bar{T}}\times\vec{\bar{N}})=\gamma(s)\cdot(\vec{T}\times\vec{N})$ (respectively, $\bar{\gamma}(s)\cdot\vec{\bar{T}}=\gamma(s)\cdot T$) iff 
  \begin{eqnarray*}
  \frac{\mu(s)}{\kappa(s)}\Big(\bar{\kappa}_n(s)-\kappa_n(s)\Big)\Big\{vaF-uaE+vbG-ubF\Big\}=0\\
  \Big(respectively,\ \frac{\mu(s)}{\kappa(s)}(av'+bu')\Big(\bar{\kappa}_n(s)-\kappa_n(s)\Big)=0\Big)\\
  \text{i.e., iff } \mu(s)=0 \text{ or }\bar{\kappa}_n(s)=\kappa_n(s).
  \end{eqnarray*}
  If $\mu(s)=0$ then from the definition of rectifying curve, we see that $\gamma(s)=\lambda(s)\vec{t}(s)$, i.e., the position vector of the rectifying curve $\gamma(s)$ is in the direction of tangent vector to itself.\\
  Otherwise if $\mu(s)\neq 0$, then $\bar{\kappa}_n(s)=\kappa_n(s)$, i.e., the normal curvature is invariant.
  \end{proof}
  \begin{cor}
  Let $f:S\rightarrow \bar{S}$ be an isometry and $\gamma(s)$ be an rectifying curve on $S$. If the component of the rectifying curve $\gamma(s)$ along $\vec{T}\times\vec{N}$ (respectively, $\vec{T}$) is invariant and the position vector of $\gamma(s)$ is not in the direction of the tangent vector to $\gamma(s)$, then $\gamma(s)$ is asymptotic iff $\bar{\gamma}(s)$ is asymptotic.
  \end{cor}
  \begin{proof}
  From Corollary $3.2.1.$,  $\bar{\gamma}(s)\cdot(\vec{\bar{T}}\times\vec{\bar{N}})=\gamma(s)\cdot(\vec{T}\times\vec{N})$ (respectively  $\bar{\gamma}(s)\cdot\vec{\bar{T}}=\gamma(s)\cdot T$) and the position vector of $\gamma(s)$ is not in the direction of the tangent vector to $\gamma(s)$ iff
  $\kappa_n(s)=\bar{\kappa}_n(s)$.\\
  Therefore $\gamma(s)$ is asymptotic iff $\kappa_n(s)=0$, i.e., iff $\bar{\kappa}_n(s)=0$, i.e., iff $\bar{\gamma}(s)$ is asymptotic.
  \end{proof}

\section{osculating curves on a surface}
 The equation of an osculating curve on a surface patch $\phi(u,v)$ of $S$ is given by 
  \begin{equation}\label{o1}
 \alpha(s)=\lambda_1(s)(\phi_uu'+\phi_vv')+\frac{\lambda_2(s)}{\kappa(s)}\Big\{(\phi_{uu}u'^2+2\phi_{uv}u'v'+\phi_{vv}v'^2)+(\phi_uu''+\phi_vv'')\Big\},
 \end{equation}
 for some functions $\lambda_1(s)$ and $\lambda_2(s)$.\\
 \begin{thm}
 Let $f:S\rightarrow\bar{S} $ be an isometry. If $\alpha$ and $\bar{\alpha}$ are respectively osculating curves on $S$ and $\bar{S}$, then for the normal component of $\alpha(s)$, we have
 \begin{equation}\label{ot1}
 \bar{\alpha}\cdot\vec{\bar{N}}-\alpha\cdot\vec{N}=\frac{\lambda_2(s)}{\kappa(s)}\Big(\bar{\kappa}_n(s)-\kappa_n(s)\Big).
 \end{equation}
 \end{thm}
 \begin{proof}
From $(\ref{o1})$, we see that
\begin{eqnarray}\label{o2}
\nonumber
\alpha\cdot\vec{N}&=&\lambda_1(s)(\phi_uu'+\phi_vv')+\frac{\lambda_2(s)}{\kappa(s)}\Big\{(\phi_{uu}u'^2+2\phi_{uv}u'v'+\phi_{vv}v'^2)\\
\nonumber
&&+(\phi_uu''+\phi_vv'')\Big\}\cdot\vec{N},\\
\nonumber
&=&\frac{\lambda_2(s)}{\kappa(s)}\Big(\phi_{uu}\cdot\vec{N}u'^2+2\phi_{uv}\cdot\vec{N}u'v'+\phi_{vv}\cdot\vec{N}v'^2\Big),\\
&=&\frac{\lambda_2(s)}{\kappa(s)}\Big(Lu'^2+2Mu'v'+Nv'^2\Big).
\end{eqnarray}
Similarly for the isometry of $\alpha$ and $\vec{N}$, we have
\begin{equation}\label{o3}
\bar{\alpha}\cdot\vec{\bar{N}}=\frac{\lambda_2(s)}{\kappa(s)}\Big(\bar{L}u'^2+2\bar{M}u'v'+\bar{N}v'^2\Big).
\end{equation}
In view of $(\ref{o2})$ and $(\ref{o3})$, we obtain
\begin{eqnarray*}
 \bar{\alpha}\cdot\vec{\bar{N}}-\alpha\cdot\vec{N}&=&\frac{\lambda_2(s)}{\kappa(s)}\Big(\bar{L}u'^2+2\bar{M}u'v'+\bar{N}v'^2\Big)-\frac{\lambda_2(s)}{\kappa(s)}\Big(Lu'^2+2Mu'v'+Nv'^2\Big),\\
 &=&\frac{\lambda_2(s)}{\kappa(s)}\Big(\bar{\kappa}_n(s)-\kappa_n(s)\Big).
\end{eqnarray*}
This proves our claim.
 \end{proof}
 \begin{cor}\label{ocor1}
 Let $f:S\rightarrow \bar{S}$ be an isometry and $\alpha(s)$ be an osculating curve on $S$. Then the normal component of the osculating curve $\alpha(s)$ is invariant iff any one of the following holds:
 \begin{itemize}
 \item[(i)]the position vector of $\alpha(s)$ is in the direction of tangent vector to $\alpha$.
 \item[(ii)] The normal curvature is invariant.
 \end{itemize}
 \end{cor}
 \begin{proof}
 From $(\ref{ot1})$,  $\bar{\alpha}\cdot\vec{\bar{N}}=\alpha\cdot \vec{N}$ iff $\frac{\lambda_2(s)}{\kappa(s)}\Big(\bar{\kappa}_n(s)-\kappa_n(s)\Big)=0$,
 \begin{equation*}
 \text{i.e., iff } \lambda_2(s)=0 \text{ or }\bar{\kappa}_n(s)=\kappa_n(s).
 \end{equation*}
 If $\lambda_2(s)=0$ then from the definition of osculating curve, we have $\alpha(s)=\lambda_1(s)\vec{t}(s)$, i.e., the position vector of the osculating curve $\alpha(s)$ is in the direction of tangent vector to itself.\\
 Otherwise if $\lambda_2(s)\neq 0$, then $\kappa_n(s)=\bar{\kappa}_n(s)$, i.e., the normal curvature is invariant.
 \end{proof}
 \begin{cor}
 Let $f:S\rightarrow \bar{S}$ be an isometry and $\alpha(s)$ be an osculating curve on $S$. If the normal component of the osculating curve $\alpha(s)$ is invariant and the position vector of $\alpha(s)$ is not in the direction of the tangent vector to $\alpha$, then $\alpha(s)$ is asymptotic iff $\bar{\alpha}(s)$ is asymptotic.
 \end{cor}
 \begin{proof}
 From Corollary $4.1.1.$,  $\bar{\alpha}(s)\cdot\vec{\bar{N}}=\alpha(s)\cdot \vec{N}$ and the position vector of $\alpha(s)$ is not in the direction of the tangent vector to $\alpha(s)$ iff $\kappa_n(s)=\bar{\kappa}_n(s)$.
 Therefore $\alpha(s)$ is asymptotic iff $\kappa_n(s)=0$, i.e., iff $\bar{\kappa}_n(s)=0$, i.e., iff $\bar{\alpha}(s)$ is asymptotic.
 \end{proof}
 \begin{thm}
Let $f:S\rightarrow\bar{S} $ be an isometry. If $\alpha$ and $\bar{\alpha}$ are respectively osculating curves on $S$ and $\bar{S}$, then the component of the osculating curve $\alpha(s)$ along $\vec{T}\times\vec{N}$ is invariant, i.e., $\alpha\cdot(\vec{T}\times\vec{N})=\bar{\alpha}\cdot(\vec{\bar{T}}\times\vec{\bar{N}})$
 \end{thm}
 \begin{proof}
 Since $f$ is isometry, for the surface patch $\phi(u,v)$ of the surface $S$, $f\circ\phi(u,v)$ is also a surface patch for the surface $\bar{S}$. Let the coefficients of first fundamental form of $\phi$ and $f\circ\phi=\bar{\phi}$ be $\{E,F,G\}$ and $\{\bar{E},\bar{F},\bar{G}\}$ respectively. Then
 \begin{equation}\label{fff1}
 G=\bar{G},\ F=\bar{F} \text{ and } E=\bar{E},
 \end{equation}
 and 
 \begin{equation}\label{fff2}
 \bar{E}_u=E_u,\ \bar{E}_v=E_v,\ \bar{F}_u=F_u,\ \bar{F}_v=F_v,\ \bar{G}_u=G_u,\text{ and }\bar{G}_v=G_v.
 \end{equation}
 Now 
 \begin{equation}\label{fff3}
 \begin{cases}
 \phi_{uu}\cdot\phi_u=\frac{1}{2}E_u,\  \phi_{uv}\cdot\phi_u=\frac{1}{2}E_v,\  \phi_{uu}\cdot\phi_v=F_u-\frac{1}{2}E_v,\\
 \phi_{uv}\cdot\phi_v=\frac{1}{2}G_u,\  \phi_{vv}\cdot\phi_v=\frac{1}{2}G_v,\  \phi_{vv}\cdot\phi_u=F_v-\frac{1}{2}G_u.
 \end{cases}
 \end{equation}
 From $(\ref{o1})$ and $(\ref{fff3})$, we have
 \begin{eqnarray*}
 \alpha\cdot(\vec{T}\times\vec{N})&=& \alpha\cdot\{(a\phi_u+b\phi_v)\times\vec{ N}\},\\
 &=& \alpha\cdot(Fa\phi_u-Ea\phi_v+Gb\phi_u-Fb\phi_v),\\
 &=&\Big[\lambda_1(s)(\phi_uu'+\phi_vv')+\frac{\lambda_2(s)}{\kappa(s)}\Big\{(\phi_{uu}u'^2+2\phi_{uv}u'v'+\phi_{vv}v'^2)\\
 &&+(\phi_uu''+\phi_v)v''\Big\}\Big]\cdot(Fa\phi_u-Ea\phi_v+Gb\phi_u-Fb\phi_v),\\
 &=&\lambda_1(s)(av'-bu')(F^2-EG)+\frac{\lambda_2(s)}{\kappa(s)}\Big[u'^2a\Big\{\frac{1}{2}FE_u-E\Big(F_u-\frac{1}{2}E_v\Big)\Big\}\\
 &&+u'^2b\Big\{\frac{1}{2}GE_u-F\Big(F_u-\frac{1}{2}E_v\Big)\Big\}+2u'v'a\Big\{\frac{1}{2}FE_u-\frac{1}{2}EG_u\Big\}2u'v'b\Big\{\frac{1}{2}GE_u\\
 &&-\frac{1}{2}FG_u\Big\}+v'^2a\Big\{F\Big(F_v-\frac{1}{2}G_u\Big)-\frac{1}{2}EG_v\Big\}-v'^2b\Big\{G\Big(F_v-\frac{1}{2}G_u\Big)-\frac{1}{2}FG_v\Big\}\Big].
 \end{eqnarray*}
 By the virtue of $(\ref{fff1})$, $(\ref{fff2})$ and $(\ref{fff3})$, the last equation yields 
 \begin{eqnarray*}
 \alpha\cdot(\vec{T}\times\vec{N})&=&\lambda_1(s)(av'-bu')(\bar{F}^2-\bar{E}\bar{G})+\frac{\lambda_2(s)}{\kappa(s)}\Big[u'^2a\Big\{\frac{1}{2}\bar{F}\bar{E}_u-\bar{E}\Big(\bar{F}_u-\frac{1}{2}\bar{E}_v\Big)\Big\}\\
 &&+u'^2b\Big\{\frac{1}{2}\bar{G}\bar{E}_u-\bar{F}\Big(\bar{F}_u-\frac{1}{2}\bar{E}_v\Big)\Big\}+2u'v'a\Big\{\frac{1}{2}\bar{F}\bar{E}_u-\frac{1}{2}\bar{E}\bar{G}_u\Big\}2u'v'b\Big\{\frac{1}{2}\bar{G}\bar{E}_u\\
 &&-\frac{1}{2}\bar{F}\bar{G}_u\Big\}+v'^2a\Big\{\bar{F}\Big(\bar{F}_v-\frac{1}{2}\bar{G}_u\Big)-\frac{1}{2}\bar{E}\bar{G}_v\Big\}-v'^2b\Big\{\bar{G}\Big(\bar{F}_v-\frac{1}{2}\bar{G}_u\Big)-\frac{1}{2}\bar{F}\bar{G}_v\}\Big],\\
  &=&\bar{\alpha}\cdot(\vec{\bar{T}}\times\vec{\bar{N}}).
 \end{eqnarray*}
 This proves the result.
 \end{proof}

\section{acknowledgment}
 The second author greatly acknowledges to The University Grants Commission, Government of India for the award of Junior Research Fellow.

\end{document}